\documentclass[12pt]{amsart}

\usepackage{a4}
\usepackage{amssymb}
\usepackage{amsmath}
\usepackage{amsthm}
\usepackage{amstext}
\usepackage{amscd}
\usepackage{latexsym}
\usepackage{graphics}
\usepackage{color}
\theoremstyle{plain}
\newtheorem{thm}{Theorem}[section]

\newtheorem{lemma}[thm]{Lemma}
\newtheorem{cor}[thm]{Corollary}
\newtheorem{prop}[thm]{Proposition}
\theoremstyle{definition}
\newtheorem{remark}[thm]{Remark}
\newtheorem{notation}[thm]{Notation}

\numberwithin{equation}{section}

\newcommand{\inj}{\hookrightarrow}
\newcommand{\tensor}{\otimes}

\newcommand{\intersection}{\cap}

\newcommand{\Spec}{{\rm Spec \,}}
\newcommand{\Gal}{{\rm Gal}}
\newcommand{\sO}{{\mathcal O}}
\newcommand{\sS}{{\mathcal S}}
\newcommand{\sX}{{\mathcal X}}
\newcommand{\sW}{{\mathcal W}}
\newcommand{\sY}{{\mathcal Y}}
\newcommand{\sZ}{{\mathcal Z}}
\newcommand{\A}{{\mathbb A}}
\newcommand{\B}{{\mathbb B}}
\renewcommand{\P}{{\mathbb P}}

\input{xy}
\xyoption{all}
\begin{document}

\title[Fundamental group of quotient singularities]{On the fundamental
group of a variety with quotient singularities}

\author[I. Biswas]{Indranil Biswas}

\address{School of Mathematics, Tata Institute of Fundamental
Research, Homi Bhabha Road, Bombay 400005, India}

\email{indranil@math.tifr.res.in}

\author[A. Hogadi]{Amit Hogadi}

\address{School of Mathematics, Tata Institute of Fundamental
Research, Homi Bhabha Road, Bombay 400005, India}

\email{amit@math.tifr.res.in}



\date{}

\begin{abstract}
Let $k$ be a field, and let $\pi\, :\, \widetilde{X}\,\longrightarrow\, X$ be a proper
birational morphism of irreducible $k$--varieties, where $\widetilde{X}$ is smooth and $X$
has at worst quotient singularities. When the characteristic of $k$ is zero,
a theorem of Koll\'ar in \cite{Ko} says
that $\pi$ induces an isomorphism of \'etale fundamental groups. We give a proof of
this result which works for all characteristics. As an application, we prove that 
for a smooth projective irreducible surface $X$ over an algebraically closed field $k$,  
the \'etale fundamental group of the Hilbert scheme of
$n$ points of $X$, where $n\, >\, 1$, is canonically isomorphic to the abelianization of
the \'etale fundamental group of $X$. Koll\'ar has pointed out how the proof
of the first result can be extended to cover the case of quotients by finite group schemes.
\end{abstract}

\maketitle

\section{Introduction}

Following is the main result proved here:

\begin{thm}\label{quotient}
Let $k$ be any field and $X/k$ a connected variety which \'etale locally can be expressed as a
quotient of a smooth variety by the action of a finite group of
automorphisms of it. Let
$$
\pi\,:\,\widetilde{X}\,\longrightarrow \, X
$$
be a proper birational morphism with $\widetilde{X}$ smooth
and connected. Then for any geometric
point $\widetilde{x}_0$ of $\widetilde{X}$, the induced homomorphism 
$$
\pi_*\, :\, \pi_1^{et}(\widetilde{X},\widetilde{x}_0)\,\longrightarrow\,
\pi_1^{et}(X,\pi(\widetilde{x}_0))
$$
is an isomorphism.
\end{thm} 

When the characteristic of $k$ is zero, this is a theorem of Koll\'ar 
\cite[p. 203, Theorem (7.5.2)]{Ko}. Theorem \ref{quotient} is proved in 
Section \ref{se3}. J\'anos Koll\'ar has pointed out, \cite{ko2}, that the
proof of Theorem \ref{quotient} can be extended to establish a 
stronger result that addresses the more general situation of quotient
of smooth varieties by finite group schemes (see Theorem \ref{generalcase}).

Let $X$ be a connected 
smooth projective surface over an algebraically closed field $k$. For any
positive integer $n$, let ${\rm Hilb}^n(X)$ denote the Hilbert
scheme of $n$ points on $X$. This
${\rm Hilb}^n(X)$ is an irreducible smooth projective variety of
dimension $2n$ (\cite[p. 517, Proposition 2.3]{Fo} and \cite[p. 517, Theorem 2.4]{Fo}).
The motivation of this paper was to calculate the fundamental group
of ${\rm Hilb}^n(X)$.  When $k$ is of characteristic zero, the fundamental group
of ${\rm Hilb}^n(X)$, where $n$ is at least two, 
is already known to be isomorphic to the abelianization of the
fundamental group of $X$. The proof of this (for
characteristic zero) consists of two steps. First, one shows that the fundamental group of
${\rm Sym}^n(X)$, the $n$-th symmetric product of $X$, is the abelianization of that
of $X$ provided $n>1$. This is in fact a purely topological result (cf. Theorem 5.1
and Example 5.2 of \cite{KT}; see also \cite{DT}). Now ${\rm Hilb}^n(X)$ is a
desingularization of ${\rm Sym}^n(X)$ \cite[p. 518, Corollary 2.6]{Fo}.
The second step of the proof is provided by the above
mentioned result of Koll\'ar in \cite{Ko}.

The following is an analogue of the topological result mentioned earlier.

\begin{thm}\label{sym}
Let $k$ be an algebraically closed field, and let
$X/k$ be a variety which is integral and proper. For fixed $n\, >\,
1$, let $x_0$ and $\underline{x}$ denote geometric
points of $X$ and ${\rm Sym}^n(X)$ respectively. Then $\pi_1^{et}({\rm Sym}^n(X),
\underline{x})$ is canonically isomorphic to the abelianization of $\pi_1^{et}(X,x_0)$. 
\end{thm}

Theorem \ref{sym} is proved at the end of Section \ref{se4}.

Putting Theorem \ref{quotient} and Theorem \ref{sym} together, one gets the following.

\begin{cor}
Let $X$ be a connected smooth projective surface over an algebraically closed field $k$. For any $n\, >\, 1$, the fundamental
group of ${\rm Hilb}^n(X)$ is canonically isomorphic to the abelianization of the
fundamental group of $X$. 
\end{cor}

\medskip
\noindent
\textbf{Acknowledgements.}\, We thank J\'anos Koll\'ar for his comments on an
initial version. Subsequently, he pointed out how Theorem \ref{quotient} could be generalized
\cite{ko2}. We are very grateful to him for this. We are very grateful to the
referee for helpful comments. The proof of Theorem \ref{generalcase} was substantially
simplified by the referee. The first--named author is supported by the
J. C. Bose Fellowship.

\section{Fundamental group of a desingularization of the coarse moduli space}\label{sec2}

The goal of this section is to give a description of the
fundamental group of a desingularization of the coarse moduli 
space of a smooth Deligne--Mumford stack (see Lemma \ref{smoothnoohi}).

We start by recalling some notation and state the main result,
\cite[Section 7]{noohi}, of Noohi in \cite{noohi}.
Let $\sX/k$ be a connected separated Deligne--Mumford stack of finite
type over a field $k$. Let $$\pi\,:\,\sX\, \longrightarrow\, X$$ be its coarse moduli space. For
every geometric point $q\,:\,\Spec(L)\, \longrightarrow\, \sX$, one has an
induced map
$$q\,:\,[\Spec(L)/G_q]\, \longrightarrow\, \sX\, ,$$
where $G_q$ is the isotropy group at $q$. When it is necessary to make an explicit
reference to $\sX$, we will write $G_{q,\sX}$ instead of $G_q$. Since the
\'etale fundamental group of
$[\Spec(L)/G_q]$ at the canonical point $\Spec(L)\, \longrightarrow\,
[\Spec(L)/G_q]$ is $G_q$, we have an induced homomorphism 
\begin{equation}\label{gq}
\phi_q\, :\, G_q\, \longrightarrow\, \pi^{et}_1(\sX,q)\, .
\end{equation}
The stack $\sX$ being connected, for any other geometric point $p$, one has an isomorphism
$$\pi_1^{et}(\sX,p)\, \stackrel{\sim}{\longrightarrow}\, \pi_1^{et}(\sX,q)$$
resulting in an equivalence class homomorphisms
\begin{equation}\label{re1}
\phi_q^{p}\,:\,G_q\, \longrightarrow\, \pi_1^{et}(\sX,p)\, ,
\end{equation}
where two homomorphisms are equivalent if they differ
by an inner automorphism of $\pi_1^{et}(\sX,p)$. In other words, we get a
homomorphism $\phi_q^{p}$ well defined up to conjugation.

Let $N_p$ denote the closed normal subgroup of $\pi_1^{et}(\sX,p)$ generated by the
images of $\phi_q^p$ for all geometric points $q$ of $\sX$. 

\begin{thm}[\cite{noohi}]\label{noohi}
Let $\sX\, ,\pi$ and $X$ be as above. For a geometric point $p$ of $\sX$, the corresponding point
of $X$ is also denoted by $p$. There is a natural sequence of groups
$$ 1\, \longrightarrow\, N_p \, \longrightarrow\, \pi_1^{et}(\sX,p)
\, \longrightarrow\,\pi_1^{et}(X,p)\, \longrightarrow\, 1$$
which is exact.
\end{thm}

\begin{remark}\label{coarsenoohi}
We continue with the above notation. Let $\sW\,\longrightarrow\, \sX$ be a 
Galois \'etale cover such that $N_p$ acts trivially on $\sW$. Then the 
above theorem of Noohi says that $\sW$ is induced by a Galois \'etale cover of 
$X$. It follows from property (M4) in \cite[p. 83]{noohi}
of the coarse moduli space functor that this \'etale cover of $X$
is nothing but the coarse moduli space of $\sW$. This coarse moduli space will
be denoted by $W$. Thus the following conclusion may be deduced from the above 
theorem of Noohi: For a Galois \'etale cover $\sW\,\longrightarrow\, \sX$, the induced
map $W\,\longrightarrow\, X$ on the
coarse moduli spaces is \'etale if and only if $\sW$ is the pullback of $W$ to $\sX$.
\end{remark}

For our purpose, we would also like to have an explicit relation between the 
fundamental group of a smooth Deligne--Mumford stack $\sX$ and that of a
desingularization of 
its coarse moduli space, whenever such a desingularization exists. For that purpose we set up the following notation.

\begin{notation}\label{ntildebx}
For a smooth Deligne--Mumford stack $\sX$, let $ET(\sX)$ denote the category of finite
\'etale covers of $\sX$. Let $\B_{\sX}$ be the set of all isomorphism classes of
$1$-morphisms $f\,:\,\sY\,\longrightarrow\, \sX$, where $\sY$ is a connected normal
Deligne--Mumford stack, and $f$ is proper, representable and birational. Now consider
the set $\Sigma_{\sX}$ of all pairs of the form $(\sY\stackrel{f}{\to} \sX\, ,q)$, where
$f\,\in\, \B_{\sX}$, and $q$ is a codimension one geometric point of $\sY$. For any such
pair, we have a homomorphism
\begin{equation}\label{re2}
\phi_{q,\sX}\,:\, G_{q,\sY}\,\longrightarrow\, \pi_1^{et}(\sX,x_0)
\end{equation}
defined by the composition
$$
G_{q,\sY}\,\longrightarrow\, G_{p,\sX}\,
\stackrel{\phi_p}{\longrightarrow}\, \pi_1^{et}(\sX,p)\, ,
$$
where $p\,=\,f(q)$ and $\phi_p$ is as in \eqref{gq}. We note that the
homomorphism $\phi_{q,\sX}$ in \eqref{re2} is well defined up to an inner
automorphism of $\pi_1^{et}(\sX,x_0)$ (this was explained earlier
in \eqref{re1}). For a
geometric point $x_0$ of $\sX$, let $\widetilde{N}_{x_0}$ denote the closed
normal subgroup generated by the union of the images of $\phi_{q,\sX}$ for
all pairs $(\sY,q)$ in $\Sigma_{\sX}$.
\end{notation}

\begin{lemma}\label{ff}
Let $\sZ$ be a regular algebraic stack and $i\,:\,U\,\hookrightarrow\,
\sZ$ be a dense open subset. Let  
$$i^*\,:\,ET(\sZ)\,\longrightarrow\, ET(U)$$
be the pullback functor. Then the following hold:
\begin{enumerate}
\item $i^*$ is fully faithful.
\item If the complement of $U$ in $\sZ$ has codimension at least two, then $i^*$ is an equivalence of categories. 
\end{enumerate}
\end{lemma}

\begin{proof}(1) Let $f_i\,:\,W_i\,\longrightarrow\, \sZ$, $i\,=\,1\, ,2$, be two finite
\'etale covers, and denote by $W_{iU}$ their restrictions to $U$. We need to show that the map 
\begin{equation}\label{beta}
\beta\, :\, {\rm Hom}_{\sZ}(W_1,W_2) \,\longrightarrow\, {\rm Hom}_{U}(W_{1U},W_{2U})
\end{equation}
is a bijection.
Since the maps $f_i$ are separated, it is easy to see that the
map $\beta$ in \eqref{beta} is injective. Now take any
$h\,\in\, {\rm Hom}_{U}(W_{1U},W_{2U})$. Let $\Gamma\subset W_1\times_{\sZ}W_2$ be the
closure of the graph of $h$. It is enough to show that $\Gamma$ is a graph of a morphism,
i.e., the projection $\varpi\,:\,\Gamma \,\longrightarrow\, W_1$ is an isomorphism.
That $\varpi$ is an isomorphism can be shown locally on $\sZ$.
Thus we may assume $\sZ$ and $W_i$'s are schemes. Moreover, by further \'etale
base change, we may assume $W_i$'s are trivial \'etale covers. In this case, the claim is obvious. \\

(2) By (1),  we only need to show essential surjectivity of $i^*$. In other words,
we have to show that
any finite \'etale cover $T\,\longrightarrow\, U$ extends uniquely to a  finite \'etale cover $T'\,\longrightarrow\, \sZ$. The claimed uniqueness
and descent allows us to prove the claim \'etale (or even smooth) locally on $\sZ$. Thus we may base extend to an atlas of $\sZ$ and assume that $\sZ$ is a scheme. In this case, by \cite[8.12.6]{ega4}, we 
can extend $T\,\longrightarrow\, U$ to a finite map $T'\,\longrightarrow\,
\sZ$. By normalizing $T'$, we may further assume that $T'$ is normal. By
purity of branched locus (see \cite[p. 212, X.3.1]{sga1}) it
follows that $T'\,\longrightarrow\, \sZ$ is finite \'etale.
\end{proof}

\begin{lemma}\label{highermodel}
Let $\{f\,:\,\sY\,\longrightarrow\, \sX\} \,\in\, \B_{\sX}$. Let $\sY_{sm}$ denote the
smooth locus of $\sY$. Then 
$$f^*_{ET}: ET(\sX) \,\longrightarrow\, ET(\sY_{sm})$$ is an equivalence of categories.
\end{lemma}

\begin{proof}
Since $\sX$ is smooth and $f$ is proper and birational, there exists an open 
sub-stack $U\,\subset\, \sX$ whose complement has codimension at least two and 
$f^{-1}(U)\,\stackrel{f}{\longrightarrow}\, U$ is an isomorphism. Thus $f^{-1}(U)\intersection \sY_{sm}$ is an open subset of $\sY_{sm}$ which maps isomorphically onto an open subset of $\sX$. By Lemma \ref{ff}(1), 
$$ET(\sX)\,\longrightarrow\, ET(\sY_{sm})$$
is fully faithful, and we only need to show it is essentially surjective.

Now $ET(U)\, \longrightarrow\, ET(f^{-1}(U))$ is an equivalence of categories. As $\sY$ is
normal, it follows that the complement of $\sY_{sm}$ in $\sY$ has
codimension at least two. Therefore, the composition
$$ET(\sX) \stackrel{\rm purity}{\longrightarrow} ET(U) \,\longrightarrow\, 
ET(f^{-1}(U))\stackrel{\rm purity}{\longrightarrow}
ET(\sY_{sm}\intersection f^{-1}(U))$$
is an equivalence of categories. Here the equivalences labelled by ``purity'' follow from
Lemma \ref{ff}(2). Thus if $\sW\,\longrightarrow\, \sY_{sm}$ is a finite
\'etale cover, then its restriction to $\sY_{sm}\intersection f^{-1}(U)$ is a pullback of a
finite \'etale cover of $\sW'\,\longrightarrow\, \sX$. By Lemma \ref{ff}(1), this pullback
and $\sW$ must be isomorphic, because they are isomorphic
after restricting to a dense open subset. This shows that
$ET(\sX)\, \longrightarrow\, ET(\sY_{sm})$ is essentially surjective, as required.
\end{proof}

For $\sY_1\, ,\sY_2 \,\in\, \B_{\sX}$, we denote the normalization of the unique dominant
irreducible component of $\sY_1\times_{\sX}\sY_2$ by $\left(\sY_1\times_{\sX}
\sY_2\right)^0$. 

\begin{lemma}\label{smoothnoohi}
Let $\sX/k$ be a smooth connected Deligne--Mumford stack, and let
$\sX \,\stackrel{q}{\longrightarrow} \,X$ be its coarse moduli space.
Let $\pi\,:\,\widetilde{X}\,\longrightarrow\, X$ be a proper birational morphism with
$\widetilde{X}/k$ smooth. Let $x_0$ be a geometric point of $\sX$ whose image in $X$ lies
in the open subset over which $\pi$ is an isomorphism, thus $x_0$ also defines a geometric
point of $\widetilde{X}$. Then there is a natural short exact sequence
of groups
$$ 1 \,\longrightarrow\, \widetilde{N}_{x_0}\,\longrightarrow\, \pi_1^{et}(\sX,x_0)
\,\longrightarrow\,\pi_1^{et}(\widetilde{X},x_0)\,\longrightarrow\, 1\, .$$
\end{lemma}

\begin{proof}We will prove this in three steps.\\

\noindent
\underline{Step(1)}:\, We first observe that there exists a canonical functor
$$ET(\widetilde{X})\,\longrightarrow\, ET(\sX)$$ defined as follows. Given a finite
\'etale cover of $\widetilde{X}$, restrict it to $\pi^{-1}(X_{sm})$, where $X_{sm}\,
\subset\, X$ is the smooth locus. Since $X_{sm}$ is by definition smooth, this
cover descends to define a finite \'etale cover of $X_{sm}$ and hence of
$\alpha^{-1}(X_{sm})$, where
$$
\alpha\,:\,\sX\,\longrightarrow\, X
$$
is the coarse moduli map. But
$\alpha^{-1}(X_{sm})$ is an open sub-stack of $\sX$ whose complement has codimension at least
two, and hence by purity, this extends uniquely to give a finite \'etale cover of $\sX$.
It is also easy to see that the functor $ET(\widetilde{X})\,\longrightarrow\, ET(\sX)$
defined this way sends connected \'etale covers to connected \'etale covers.
Therefore, we get a surjective group homomorphism 
\begin{equation}\label{f1}
\pi_1^{et}(\sX,x_0)\,\longrightarrow\, \pi_1^{et}(\widetilde{X},x_0)\, .
\end{equation}\\

\noindent \underline{Step(2)}:\, We now claim that $\widetilde{N}_{x_0}$ is in the kernel
of the homomorphism in \eqref{f1}.

To prove this claim, given any pair $(\sY,y_0)\,\in\, \Sigma_{\sX}$, define
$$\sY_1\,:=\, (\sY\times_X\widetilde{X})^0\, ,$$
and let $y_1$ be the unique lift of $y_0$ in $\sY_1$.
Clearly we have $\sY_1\,\in\, \B_{\sX}$. Since $\sY$ and $\sY_1$ are normal, the map
$\sY_1\,\longrightarrow\, \sY$ is an isomorphism over codimension one points
of $\sY$. Consequently, the isotropy groups at $y_0$ and $y_1$ coincide. By replacing $\sY$ by $\sY_1$, we
may assume, without any loss of any generality, that we have a $1$-morphism $g\,:\,\sY
\,\longrightarrow\, \widetilde{X}$ such that the following diagram is commutative
\begin{equation}\label{xyd}
\xymatrix{
   &  \sY\ar[dl]_{f} \ar[dr]^g  & \\
\sX\ar[dr]_\alpha    &           & \widetilde{X}\ar[dl]^{\pi} \\
     & X &   
}
\end{equation}
Let $y_0'$ be a geometric point of $\sY$ such that $g(y_0')\,=\,x_0$. 
Since $\widetilde{X}$ is a variety, the composition
homomorphism $G_{y_0}\,\longrightarrow\,
\pi_1^{et}(\sY,y_0')\,\longrightarrow\, \pi_1^{et}(\widetilde{X},x_0)$ is trivial. 
This proves the claim.\\

\noindent \underline{Step(3)}:\, To prove the theorem it remains to show that the kernel
of the homomorphism in \eqref{f1} is contained in $\widetilde{N}_{x_0}$. 

Take any $(\sY\, ,q)\,\in \,\Sigma_{\sX}$. For the geometric point $q$, represented by
$\Spec(k_q)\,\longrightarrow\, \sY$, with $k_q$ separably closed, there is a natural map $$BG_q\,=\,[\Spec(k_q)/G_q] \,
\longrightarrow\, \sY$$ and hence we get a map $BG_q\,\longrightarrow\, \sX$.
Let $ET(\sX)'$ be the full subcategory of $ET(\sX)$ consisting of finite \'etale
covers $W\,\longrightarrow\, \sX$ which have the property that for any
$(\sY\, ,q)\,\in \,\Sigma_{\sX}$, the pullback of $W\times_{\sX} BG_q \,\longrightarrow\, BG_q$ is a
trivial cover (this means that each connected component of it is isomorphic to $BG_q$). 

If $W\,\longrightarrow\, \sX$ is a Galois cover, then the above condition is equivalent to
the condition that the image of $G_q$ in $\pi_1^{et}(\sX,x_0)$ acts trivially on
$W$. Note however that the kernel of the homomorphism $\pi_1^{et}(\sX,x_0)
\,\longrightarrow\, \Gal(W/\sX)$ is a normal subgroup. And hence the condition that all $G_q$'s are contained in the normal subgroup is equivalent to the condition that $\widetilde{N}_{x_0}$ is contained in the kernel. Thus, $ET(\sX)'$, together with the fiber functor defined by the point $x_0$, is a Galois category whose associated Galois group is $\pi_1^{et}(\sX,x_0)/\widetilde{N}_{x_0}$.

Consequently, in order to prove that the kernel is contained in $\widetilde{N}_{x_0}$, it 
suffices to show that the pullback functor defined in Step (1)
$$ ET(\widetilde{X})\,\longrightarrow\, ET(\sX)$$
induces an equivalence of $ET(\sX)$ with $ET(\sX)'$. As, observed before, 
the pullback functor is fully faithful. Thus it remains to show
that its essential image coincides with $ET(\sX)'$, in other
words, any \'etale cover 
in $ET(\sX)'$ is a pullback of an \'etale cover of $\widetilde{X}$. Define
$$\sY\,:=\,(\sX\times_X\widetilde{X})^0\, .$$ Then we have
$g\,:\,\sY\,\longrightarrow\, \widetilde{X}$ and a commutative diagram as 
in \eqref{xyd}. Let $$h\,:\,\sY\,\longrightarrow\, Y$$ be the coarse moduli space. 
The morphism $g$ as in \eqref{xyd}
factors through $h$ by the definition of a coarse moduli space. 

Now, let $W\,\longrightarrow\, \sX$ be a Galois \'etale cover contained in 
$ET(\sX)'$. 
Let
$$
W_{\sY}\,=\,W\times_{\sX}\sY \,\longrightarrow\, \sY
$$
be the Galois \'etale cover. Consider the induced map
$$
\xi\, :\, W'\,\longrightarrow\, Y
$$
on the coarse moduli spaces, where $W'$ is the coarse moduli space of $W$.
We note that $\xi$ is a finite map. Let $V\,\subset\,
Y$ be the maximal open subset of $Y$ over which this map is \'etale. By
Lemma \ref{last} below, and definition of $Et(\sX)'$, this subset $V$ contains all codimension one points of $Y$. By
purity, the Galois \'etale cover $W'\vert_V\,\longrightarrow\, V$ extends to a Galois \'etale cover of the smooth locus of $Y$ and hence also to a Galois \'etale cover of $\widetilde{X}$ by Lemma \ref{highermodel}. 
\end{proof}

\begin{lemma}\label{last}
Let $\sY/k$ be an orbifold and $\sW\,\longrightarrow\, \sY$ a finite \'etale cover.
The coarse moduli spaces for $\sY$ and $\sW$ will be denoted by $Y$ and $W$ respectively.
Let $W\,\longrightarrow\, Y$ be the map induced by the \'etale cover. Let $q$ be a
geometric point of $Y$ represented by $\Spec(k_q)\,\longrightarrow\, \sY$, where $k_q$ is
separably closed. Let
$$
BG_q\,=\,[\Spec(k_q)/G_q]\,\longrightarrow\, \sY
$$
be the induced map. Then the following are equivalent:
\begin{enumerate}
\item $\sW\times_\sY BG_q \,\longrightarrow\, BG_q$ is a trivial \'etale cover (i.e.,
isomorphic to the projection from a disjoint union of copies of $BG_q$).
\item $W\,\longrightarrow\,  Y$ is \'etale over the image of $q$ in $Y$. 
\end{enumerate}
\end{lemma}

\begin{proof}
Without loss of generality, we first replace $k$ by a finite \'etale extension, and
assume that the finite group scheme $G_q$ is actually a discrete group. Hence $G_q$
is defined over $k$. The lemma may be proved by \'etale base change
on $Y$, and hence by Lemma \ref{refinedav} we may assume that
$\sY \,=\, [V/G_q]$ for some smooth variety $V/k$ with a $k_q$-point
$\widetilde{q}$ and an action of $G_q$ which fixes $\widetilde{q}$. The pullback of $\sW$
to $V$ gives
a $G_q$ equivariant \'etale cover $f\,:\,U\,\longrightarrow\, V$ such that $\sW\,=\,
[U/G_q]$. Let $\{p_1,\cdots ,p_d\}\,=\, \sW\times_{\sY}\Spec(k_q)$ denote the set of $k_q$
points of $U$ lying over $\widetilde{q}$. Note that this is a finite set and moreover $d$ is
exactly equal to the degree of the map $\sW\,\longrightarrow\, \sY$.

Now we claim that the following diagram is cartesian 
$$\xymatrix{
\left(\coprod_{i=1}^d p_i \right)/G_q \ar[d]\ar[r]^{\hspace{15mm} \coprod p_i} & \sW \ar[d] \\
BG_q  \ar[r] & \sY
}$$
This claim follows from the fact that to check that above diagram is cartesian
one further base change by the map $\Spec(k_q)\,\longrightarrow\, BG_q$.
After this, the claim follows from the definition of the set $\{p_1,\cdots ,p_d\}$.

Note that the left vertical arrow in the above diagram is a trivial \'etale cover if
and only if the action of $G_q$ on $\{p_1,\cdots ,p_d\}$ is trivial. More generally, we
observe that the number of points in $W$ lying over the point $q$ of $Y$ is precisely
equal to the set of orbits of $\{p_1,\cdots ,p_d\}$ for $G_q$ action.\\

\noindent $(2)\implies (1)$:\, If $W\,\longrightarrow\,  Y$ is \'etale over $q$,
then the number of elements lying over $q$ must match the generic degree (since it is
also finite). Thus the number of orbits of $\{p_1,\cdots ,p_d\}$ for $G_q$ action must be
exactly $d$ which is equivalent to saying that the action of $G_q$
on $\{p_1,\cdots ,p_d\}$ is trivial. \\

\noindent $(1)\implies (2)$:\, For any point $x\, \in\, U$, let $G_x$ be the
stabilizer of $x$ for the action of $G_q$. Let $S$ be the set of points $x\,\in\, U$ such
that $G_x\,=\, G_{f(x)}$, where $f$ denotes the map $U\,\longrightarrow\, V$.

We claim that $S$ is an open subset of $U$. To see this, consider the following commutative
diagram 
$$\xymatrix{
U\times G_q \ar[r]^\phi \ar[d]_{a=f\times id} &  U\times U \ar[d]^{b=f\times f}\\
V\times G_q\ar[r]^\psi & V\times V
}$$
where $\phi$, and similarly $\psi$, is the map $(u\, ,g)\,\longmapsto\, (u\, ,gu)$. If
$$\eta\,:\,U\times G_q\,\longrightarrow\, U$$ is the projection, then the set $S$ is
precisely the complement
$\eta \left( a^{-1}\psi^{-1}(\Delta_V) \backslash  \phi^{-1}(\Delta_U)\right)$, where
$\Delta_V$ and $\Delta_U$ are the diagonals.
Since $\eta$ is a closed map, and the diagonals $\Delta_V$ and $\Delta_U$ are closed, to prove
the claim it suffices to show that 
$\phi^{-1}(\Delta_U)$ is open in $a^{-1}\psi^{-1}(\Delta_V)\,=\, \phi^{-1}b^{-1}(\Delta_V)$.
Since $U\,\longrightarrow\, V$ is finite \'etale, the openness of $\phi^{-1}(\Delta_U)$
follows from the fact that $\Delta_U$ is a connected component of $b^{-1}(\Delta_V)$.
This proves the claim that $S$ is open.

Let $T\,:=\,V\backslash f(U\backslash S)$. This $T$ is an open neighborhood of $q$ having
the property that for every point $y'\,\in\, U$ lying over a point $y$ in $T$,
we have $G_y\,=\,G_{y'}$.

Now, in order to prove that $(1)\implies (2)$, we may replace $V$ by $T$, a $G_q$-invariant 
Zariski neighborhood of $q$. Thus, without loss of generality, we now assume that all 
points $y\,\in\, V$ have the property that for any point $y'\,\in\, U$ with $f(y')\,=\,y$, 
we have $G_y\,=\,G_{f(y')}$. In other words, for every point $y$ of the stack $\sY$, the action of 
$G_y$ on the fiber of $\sW\,\longrightarrow\, \sY$ is trivial. By Theorem \ref{noohi}
and Remark  \ref{coarsenoohi}, the map $W\,\longrightarrow\, Y$ is \'etale.
\end{proof}

\begin{remark}\label{onlypurity1}
We note that the only reason why one requires $\widetilde{X}$ to be
smooth in Lemma \ref{smoothnoohi} is to ensure purity, meaning any
finite \'etale cover of an open subset of $X$, whose complement has codimension
at least two,  extends to a finite \'etale cover of $\widetilde{X}$. Thus the
statement of Lemma \ref{smoothnoohi} remains true if  
$\widetilde{X}$ is normal and satisfies purity. 
\end{remark}

\section{Group actions and Blow-ups}\label{se3}

The goal of this section is to prove Theorem \ref{quotient}. 
The proof uses the following observation on the tangent spaces to the
points of a blow-up.

\begin{prop}\label{blowup}
Let $k$ be an algebraically closed field, and let $C$ be
a cyclic group of prime power order acting faithfully on a smooth variety $X/k$. Assume
that there exists a closed point $p\,\in\, X$ fixed by $C$. Then one can find a
$C$--equivariant proper birational morphism $ \pi\,:\,Y\,\longrightarrow\, X$ and a
smooth closed point $\widetilde{p}\,\in\, Y$, such that the action of $C$
on the tangent space $T_{\widetilde{p}}Y$ is multiplication by a character of $C$.
\end{prop} 

\begin{proof}
\underline{Case (1): $|C|$ is a power of $Char(k)$.} In this case,
the group $C$ is unipotent. Let $V_0\,\subset\, T_p X$ be a maximal direct summand of
the $C$--module $T_p X$ on which $C$ acts trivially. Thus we have a decomposition 
$$ T_p X\,=\, V_0\oplus V'\, .$$ 
Note that $V_0$ could be zero.

Using induction, we will show the following: by
performing $C$--equivariant blow-ups at suitable closed points, $\dim_k V_0$
can be increased until it coincides with the dimension of $X$.

If $V'\,=\,0$, then there is nothing to prove.
Else, we can find a line $\ell\,\subset\, V'$ on which the action of $C$ is trivial.
Let $\widetilde{X}\,\longrightarrow\, X$ be the blow up of $X$ at $x_0$, and let
$\widetilde{p}\,\in\, \widetilde{X}$ be the point above $p$ corresponding to the line $\ell$. The
tangent space to $\widetilde{p}$ is canonically isomorphic to
$$ T_{\widetilde{p}}\widetilde{X} \,= \,\ell \oplus (\ell^*\tensor (T_pX/\ell))\, ,$$
where $\ell^*$ is the dual line. The group $C$ acts trivially on $\ell$ and hence
also on its dual $\ell^*$. Therefore, we have a $C$--equivariant isomorphism 
$$ T_{\widetilde{p}}\widetilde{X} \,\cong\, \ell \oplus V_0\oplus (V'/\ell) \, .$$
Note that $C$ acts trivially on the above direct summand $\ell\oplus V_0$.
Consequently, we have increased the dimension of the direct summand on
which $C$ acts trivially, and now the proof of the
claim is finished.\\

\noindent
\underline{Case(2): $|C|$ is coprime to $Char(k)$.} In this case, the group $C$
is linearly reductive, meaning every sub-representation of $C$ is a direct
summand. Let $\chi$ be a primitive character (i.e., a generator of the character group) of $C$.
For any integer $a$, let $$V_{\chi^a}\,\subset\, T_pX$$ be the isotypical component
on which $C$ acts as multiplication through $\chi^a$. There are finitely many
integers $\{a_i\}_{i=1}^r$ such that
$$
T_pX \,=\, \bigoplus_{i=1}^r V_{\chi^{a_i}}\, .
$$
Since the action of $C$ on $T_pX$ is faithful, at least one $a_i$ (say $a_1$)
is coprime to $|C|$. As $|C|$ is a prime power, we conclude that
$\chi^{a_1}$ is also a primitive character. By replacing $\chi$ by $\chi^{a_1}$, we may
assume that $a_1\,=\,1$.

If $r=1$, then there is nothing to prove. We may therefore assume that $r\,\geq\,2$.

Let $\ell \,\subset\, V_\chi\,=\, V_{\chi^1}$ be a line, and write
$$V_{\chi}\,=\,\ell \oplus V_\chi'\, .$$ By Lemma \ref{lem:subvariety}, there
exists a closed subvariety $Y\,\subset\, X$ containing $p$, and smooth at
$p$, such that $T_pY\,=\, V'_{\chi}$. Let 
$$\pi_1\,:\,X_1\,\longrightarrow\, X$$ be the blow up of $X$ along $Y$. Let
$p_1\,\in\, X_1$ be the point lying above $p$ corresponding to the normal direction
$\ell$. By Lemma \ref{lem:tangent}, we have
$$
T_{p_1}X_1 \,=\, V_\chi \oplus \ell^* \tensor \left(\bigoplus_{i=2}^r
V_{\chi^{a_i}}\right)\, .
$$
Since $V_{\chi}$ and hence $\ell$ are naturally subspaces of $T_{p_1}X_1$, we
can inductively define $X_m$, $m\, \geq\, 1$, as the blow-up of
$X_{m-1}$ along the subvariety through $p_{m-1}$ defined by $V_{\chi}'\subset T_{p_{m-1}}X_{m-1}$ (see Lemma \ref{lem:tangent}) with $p_m$
being the point lying above $p_{m-1}$ corresponding to the normal direction $\ell$. 
Thus the tangent space at $p_m$ decomposes as
$$ T_{p_m}X_m \,=\, V_{\chi}\oplus \left( \bigoplus_{i=2}^r
V_{\chi^{a_i-m}}\right)\, .$$
Therefore, if $m$ is an integer such that $$a_2-m\,\equiv\, 1 \ {\rm mod} \ |C|\, ,$$
we see that the number of characters of $C$ which have non-zero eigenspaces
in $T_{p_m}X_m$ is strictly less than $r$. Now, by induction, we may replace
$X$ with a higher birational model and assume that $r\,=\,1$. This completes
the proof of the proposition.
\end{proof}

\begin{lemma}\label{lem:subvariety}
Continuing with the notation of Proposition \ref{blowup}, assume that $Char(k)$
does not divide $|C|$. Then given any $C$--invariant subspace 
$W\,\subset\, T_p(X)$, there exists a $C$--invariant closed subvariety $Y\,
\subset\, X$, containing $p$ and smooth at $p$, such that $W\,=\,T_pY$. 
\end{lemma}

\begin{proof}
Since $C$ is linearly reductive, we can find a $C$--equivariant decomposition 
$$ T_pX \,=\, W \oplus W_1 \, .$$
This gives a $C$--equivariant decomposition of the cotangent space 
$$m_p/m_p^2\,=\,(T_pX)^*\,=\, W^*\oplus W_1^*\, ,$$ 
where $m_p$ is the maximal ideal of the local ring $\sO_{X,p}$ at $p$. In order
to prove the lemma we may assume that $X\,=\,\Spec(A)$ is an affine scheme, by
looking at a $C$--equivariant affine neighborhood of $p$.

The maximal ideal of $A$ corresponding to $p$ will also be denoted by
$m_p$; the distinction will be clear from the context. We have a surjection 
$$ A \,\longrightarrow\, A/m_p^2\, . $$
This is a $C$--equivariant homomorphism and it admits a $C$--equivariant splitting.
Therefore, if $r\,=\,{\rm codim}_{k(p)}\,W$, then there exist $f_1,\cdots ,f_r\,\in
\, A$ such that the $k$--span of $\{f_i\}_{i=1}^r$ is $C$--invariant and it maps
isomorphically onto $W_1^*$. Now let $Y$ be the closed subvariety defined by
the ideal generated by $\{f_i\}_{i=1}^r$. If $m_{p,Y}$ is the ideal of the local ring of
$p$ at $Y$, then 
$$m_{p,Y}/m_{p,Y}^2 \,=\, W^*\, .$$
Taking duals, we see that $T_pY = W$ as required. To show $Y$ is smooth at $p$, we first note that by 
Krull's principal ideal theorem, we have ${\rm codim}(Y)\,\leq\, r$. However, ${\dim}_{k(p)}(m_{p,Y}/m_{p,Y}^2)\,=\,r$. This completes the proof.
\end{proof}

\begin{lemma}\label{lem:tangent}
Let $G$ be a finite group acting on $X$ with a closed fixed point $p$. Assume
that $|G|$ is coprime to $Char(k)$. Let $Y\,\subset\, X$ be a closed, smooth,
$G$--invariant subvariety containing $p$, and let
$$\pi\,:\,\widetilde{X}\,\longrightarrow\, X$$
be the blow-up of $X$ along $Y$.
Take a line $\ell \,\subset\, T_pX$ not contained in $T_pY$. Let $\widetilde{p}
\,\in\, \widetilde{X}$
be the point above $p$ corresponding to the normal direction given
by $\ell$. Then there exists a $G$--equivariant isomorphism 
$$ T_{\widetilde{p}}\widetilde{X} \,\cong\, \left(T_pY\oplus \ell\right)\oplus
\left(\ell^*\tensor (T_pX/(T_pY+\ell))\right)\, .$$
\end{lemma}

\begin{proof}
We observe that the assertion depends only on a formal neighborhood of $p$ in
$X$. The strategy of the proof consists of the following two steps:
\begin{itemize}
\item Show that formally locally at $p$, the triple $(X\, ,Y\, ,p)$ is $G$
equivariantly isomorphic to $(\A^n_k\, ,\A^m_k\, ,\underline{0})$, where $G$ acts
linearly on $\A^n_k$ with $\A^m_k\,\subset\, \A^n_k$ being a
$G$--invariant subspace, and

\item check the claim of the lemma by using the explicit construction
of the blow up as a closed subvariety of $\A^n_k\times_k \P^{n-m-1}_k$.
\end{itemize}
The later step is straightforward and we omit the details.

Without loss of any generality, we assume that $X\,=\,\Spec(A)$ is affine. The
maximal ideal corresponding to $p$ will be denoted by $m_p$.
Define $$V \,:=\, m_p/m_p^2\  \text{ and }\ 
W \,:=\, m_{p,Y}/m_{p,Y}^2\,\subset\, V\, ,$$ where
$m_{p,Y}$ is the maximal ideal of $p$ in the local ring of $Y$. Since $G$ is
linearly reductive, there exists a $G$--equivariant splitting of the epimorphism 
$$ A \,\longrightarrow\, A/m_p^2\, . $$
Thus we get an $n$--dimensional $k$--subspace $V_0\,\subset\, A$
that maps isomorphically onto $V$. We denote by $W_0$ this subspace of $V_0$
mapping isomorphically onto $W$. The canonical $k$--algebra homomorphism
$$ {\sS}ym^*_k(V_0) \,\longrightarrow\,  A $$
gives the formal local isomorphism with the affine space
$\Spec({\sS}ym^*(V_0))$ that we are seeking.
\end{proof}

\begin{proof}[Proof of Theorem \ref{quotient}]
In order to prove the theorem, we may replace $k$ by its algebraic closure inside the
function field of $X$, and assume that $X$, and hence $\widetilde{X}$, are geometrically connected.
Fix a separable closure $K/k$, and let $X_K$ (respectively, $\widetilde{X}_K$) denote the
base extension to $K$ of $X$ (respectively, $\widetilde{X}$). We lift the geometric point
$x_0$ to the corresponding base extended varieties. Therefore, we have a commutative diagram
with exact rows:
$$
\xymatrix{
1 \ar[r] &  \pi_1^{et}(\widetilde{X}_K,x_0)\ar[r]\ar[d] & \pi_1^{et}(\widetilde{X},x_0)\ar[r]\ar[d] & \Gal(K/k) \ar[r] \ar@{=}[d] & 1 \\
1 \ar[r] &  \pi_1^{et}(X_K,x_0)\ar[r] & \pi_1^{et}(X,x_0)\ar[r] & \Gal(K/k) \ar[r] & 1
}$$
It is now clear that to prove that the middle vertical arrow is an isomorphism, it
is enough to prove that the left-most vertical arrow is an isomorphism. Thus we may
base extend to $K$ and assume that our field is separably closed. Further, if $\overline{K}/K$ is an algebraic closure of $K$, then as $K$ is separably closed, this extension is purely inseparable. Base extension by an inseparable extension does not change the \'etale fundamental group. Hence we may base extend to the algebraic closure of $K$ and henceforth assume that $K$ is in fact an algebraically closed field.   \\

The theorem is equivalent to the statement that the pullback functor 
$$ ET(X)\,\longrightarrow\, ET(\widetilde{X})$$
(see Notation \ref{ntildebx}) is an equivalence of categories. This
question is \'etale local on $X$ and hence we may assume that $X$ is actually a
quotient of a variety $Y$ by a finite group $G$ acting faithfully on $Y$. Let
$$\sX\,=\, [Y/G]$$ be the quotient stack. Fix a point $x_0$ of $\sX$. By Theorem
\ref{noohi} and Lemma \ref{smoothnoohi}, it is enough to show that the natural
inclusion
$$ \widetilde{N}_{x_0} \,\subset\,N_{x_0} $$
(see Notation \ref{ntildebx}) is surjective. This is equivalent to the statement
that for any point $p$ of $\sX$, the images of $\phi_{q,\sX}$ (see \eqref{re2})
for all $(\sY\, ,q)\,\in\, \Sigma_{\sX}$ such that $q$ lies over $p$,
together generate the isotropy group $G_p$ at $p$. We will prove this statement
by showing that all elements of $G_p$ of prime power order are in
the image of $\phi_{q,\sX}$ for some $(\sY\, ,q)$.

Let $x\,\in\, G_p$ be an element of prime power order, and
let $C$ be the cyclic subgroup of $G_p$ generated by $x$. It suffices to
find a $(\sY\, ,q)\,\in\,\Sigma_{\sX}$ (see Notation \ref{ntildebx}) such
that image of the induced homomorphism on
isotropy groups $$ G_{q}\,\longrightarrow\, G_{p}$$ is $C$. 
By Lemma \ref{lem:tangent}, there exists a $C$--equivariant proper birational
morphism $\pi\,:\,\widetilde{Y}\,\longrightarrow\, Y$ and a point
$\widetilde{p}\,\in\, \widetilde{Y}$ with $\pi(\widetilde{p})\,=\, p$, such that
the action of $C$ on the tangent space $T_{\widetilde{p}}\widetilde{Y}$ is
multiplication by a character.

Let $Z\,\longrightarrow\, \widetilde{Y}$ be the blow-up at
$\widetilde{p}$, and let $q$ denote the generic point of the exceptional
divisor for this blow-up of $\widetilde{Y}$.
Then clearly $C$ acts trivially on this exceptional
divisor, and hence the isotropy group of $q$ is $C$. For every $g\,\in\, G$,
define $g^*Z$ by the following Cartesian diagram
$$ \xymatrix{
g^*Z \ar[r] \ar[d] & Z \ar[d] \\
Y \ar[r]^g & Y
}
$$
Each $g^*Z$ is a $Y$--scheme. Consider the product $$P\,=\,\prod_{g\in G}g^*Z$$ in the
category of $Y$-schemes, i.e., it is a fiber product over $Y$. 

We have a natural action of the group $G$ on $P$ which permutes
the factors. Since each map $g^*Z\,\longrightarrow\,Y$ is proper and
birational, there exists a unique dominant irreducible component of $P$. We
denote the normalization of this component by $T$. By construction, $\pi_T:T\,
\longrightarrow\, Y$ is a $G$--equivariant proper birational morphism. \\

Let $f\,:\,T\,\longrightarrow\, Z$ be the restriction of
the natural projection map 
$$P \,\longrightarrow\, e^*Z \,=\, Z\, ,$$
where $e\,\in\, G$ is the identity element.
We claim that $f$ is $C$--equivariant.

To prove this claim, we first observe that there
exists a $G$--invariant open subset $U$ of $Y$ over which  both
$$
\pi_T\,:\,T
\,\longrightarrow\, Y ~\ \text{ and }~\ \pi_Z\,:\,Z\,\longrightarrow\, Y
$$
are isomorphisms. Clearly, the restriction of $f\,:\,T\,\longrightarrow\, Z$ to $U$ is
$C$--equivariant. But $T$ is a separated integral variety, and $U$ is dense in $T$.
Consequently, $f$ is also $C$-equivariant, proving the claim. \\

Since $Z$ is normal, the morphism $f$ is an isomorphism over all codimension one
points of $Z$, in particular over $q$. Recall that $q$ is the codimension one point of
$Z$ whose isotropy group is whole of $C$. Hence there exists a unique codimension one
point $\widetilde{q}$ of $T$ lying over $q$. Since $T\,\longrightarrow\, Z$ is
$C$--equivariant, we conclude that $C$ is contained in the isotropy group of $q$. This
completes the proof of Theorem \ref{quotient}.
\end{proof}

\begin{remark}\label{onlypurity2}
In view of Remark \ref{onlypurity1}, the statement of Theorem \ref{quotient}
continues to hold for all normal varieties $\widetilde{X}$ which satisfy purity. 
\end{remark}

Below we mention Koll\'ar's observation which handles the general case of quotient of a 
smooth variety by a finite group scheme.  The argument to deduce the theorem below from 
Theorem \ref{quotient} was outlined to us by Koll\'ar in \cite{ko2}. The proof given 
below is based on a significant simplification of this argument given by the referee.

\begin{thm}[\cite{ko2}]\label{generalcase}
Let $k$ be any field, and let $G/k$ be a finite group scheme acting faithfully on a smooth
variety $Y$. If $X$ is the geometric quotient of $Y$ by $G$, and $f\,:\,
\widetilde{X}\,\longrightarrow\, X$ is any proper birational morphism with
$\widetilde{X}$ smooth, then the homomorphism
$$
\pi_1^{et}(\widetilde{X},x) \,\longrightarrow\, \pi_1^{et}(X,f(x))
$$
is an isomorphism for any geometric point $x$ of $\widetilde{X}$.  
\end{thm}

\begin{proof}
Without any loss of generality we may assume that $k$ is perfect (or even algebraically
closed) by base extending to the algebraic closure of $k$. Thus the reduced subscheme
$G_{\rm red}\,\subset\, G$ is a subgroup. 

Consider the quotient stack $\sX\,:=\,[Y/G_{red}]$. Let $\sX\,\longrightarrow\, X'$
be its coarse moduli space. Take $\widetilde{X}$ as in the theorem.
Let $\widetilde{X}'$ be the normalization of  $\widetilde{X}$ in the function
field of $X'$. We thus have the following commutative diagram where horizontal arrows
are proper birational and vertical arrows are universal homeomorphisms.
$$\xymatrix{
\widetilde{X}' \ar[r]^\tau \ar[d] & X' \ar[d] \\
\widetilde{X} \ar[r] & X
}$$
Thus in order to prove the theorem, it is enough to show that the induced homomorphism
$$
\tau_* \, :\, \pi_1^{et}(\widetilde{X}', z)\,\longrightarrow\, \pi_1^{et}(X', \tau(z))\, .
$$
We note that $\widetilde{X}'$ is a normal variety satisfying purity
because $\widetilde{X}$ is smooth (see Remark \ref{onlypurity1}). Now from
Theorem \ref{quotient} and Remark \ref{onlypurity2} it follows that $\tau_*$ is
an isomorphism.
\end{proof}

\section{Fundamental group of the symmetric product}\label{se4}

Let $S^n$ be the group of permutations of $\{1\, ,\cdots \, ,n\}$. The following lemma will be used to prove Theorem \ref{sym}.

\begin{lemma}
Let $G$ be any group, and $G^n \,:=\, \prod_{i=1}^nG$. For the natural action
on $G^n$ of $S^n$, consider the semi-direct product $G^n\rtimes S^n$. Let $N$ be the normal subgroup of
$G^n\rtimes S^n$ generated by the image of $S^n$ for the natural injection
$S^n\,\longrightarrow\, G^n\rtimes S^n$. Then the quotient 
$(G^n\rtimes S^n)/N$ is canonically isomorphic to $G^{ab}\,:=\, G/[G\, ,G]$, the
abelianization of $G$.
\end{lemma}

\begin{proof}
The group operation of $G^n\rtimes S^n$ is given by 
$$ (g_1,\cdots ,g_n,\sigma)\cdot (g_1',\cdots ,g_n',\sigma')\,= \,
(g_1 g_{\sigma(1)},\cdots ,g_n g_{\sigma(n)},\sigma \sigma')\, .$$
We have the following relation modulo $N$:
$$ (g_1,\cdots ,g_n,\gamma)\,=\, (g_{\sigma(1)},\cdots ,g_{\sigma(n)},\gamma)
 \ \ \ {\rm mod} \ N \ , \forall \ \sigma,\gamma \in S^n\, .$$
Thus 
\begin{equation}\label{e2}
(g_1,\cdots ,g_n,\sigma) \,= \,\prod_{i=1}^n (1,\cdots ,g_i,\cdots ,1,\sigma)
\,=\, \prod_{i=1}^n(g_i,1,\cdots ,1,\sigma)
\end{equation}
$$
\,=\, (g_1\cdots g_n,1\, ,\cdots ,1,\sigma)\, .
$$
Also, since $$(g,1,\cdots ,1,\sigma)\cdot (h,1,\cdots ,\sigma)
\,=\, (gh,1,\cdots ,1,\sigma)\,=\, (hg,1,\cdots ,1,\sigma)\, ,$$
one can deduce that
\begin{equation}\label{e1}
(g,1,\cdots ,1\, ,\sigma) \,=\, 1 \ {\rm mod} \ N
\end{equation}
if $g$ belongs to the commutator subgroup $[G,G]\,\subset\, G$. 
Define the homomorphism
$$
\phi\,:\,G^n\rtimes S^n \,\longrightarrow\, G^{ab}\, ,
(g_1,\cdots ,g_n\, ,\sigma)\,\longmapsto\, \overline{g_1}\cdots \overline{g_n}\, .
$$
Clearly $N\,\subset\, {\rm kernel}(\phi)$.

We claim that $N\,=\, {\rm kernel}(\phi)$. To prove this claim,
take any $(g_1,\cdots ,g_n,\sigma)\,\in\, {\rm kernel}(\phi)$. 
By the definition of $\phi$,
$$
\prod_{i=1}^ng_i\,\in\, [G,G]\, .
$$
{}From \eqref{e2} and \eqref{e1},
$$ (g_1,\cdots ,g_n,\sigma) \,=\,
(\prod_{i=1}^ng_i, 1,\cdots ,1,\sigma) \,=\, 1 \ {\rm mod} \ N\, .$$
This proves the claim. The claim completes the proof of the lemma.
\end{proof}

Henceforth, we assume that the field $k$ is algebraically closed.

\begin{lemma}\label{sn}
Let $X$ be a a proper, integral $k$ variety.  
Let $\sX \,:=\, [X^n/S^n]$ for the natural permutation action of $S^n$ on $X^n$. Let $x$ be a geometric point of 
$\sX$ of the form $$ x \,=\, (x_0,\cdots ,x_0)\, ,$$ where $x_0$ is a geometric point of $X$. Then
$\pi_1^{et}(\sX,x)$ is canonically isomorphic to the semi-direct product
$\pi_1^{et}(X,x_0)^n\rtimes S^n$ for the natural action of $S^n$ on $\pi_1^{et}(X,x_0)^n$. 
\end{lemma}

\begin{proof}
The group $\pi_1^{et}(X^n,x)$ is canonically identified with $\pi_1^{et}(X,x_0)^n$
because $X$ is proper and integral (see \cite[X.1.3 and X.1.7]{sga1}).
Since the quotient map $X^n\,\longrightarrow\, \sX$ is a Galois \'etale cover with Galois
group $S^n$, we have a short exact sequence 
$$ 1 \,\longrightarrow\, \pi_1^{et}(X,x_0)^n \,\longrightarrow\, \pi_1(\sX,x)
\,\longrightarrow\, S^n\,\longrightarrow\, 1\, . $$
Moreover, the isotropy group of the point $x$ in $X^n$ is precisely $S^n$, thus yielding a representable $1$-morphism 
$$ [x/S^n] \,\longrightarrow\, \sX$$
which induces a right splitting
$$ \pi_1^{et}([x/S^n],x)\,=\,S^n \,\stackrel{\theta}{\longrightarrow}\, \pi_1^{et}(\sX,x)$$ of the above short exact
sequence.

It remains to show that for any $\sigma\,\in\, S^n$, the conjugation
$\pi_1^{et}(X,x)^n$ by $\theta(\sigma)$ coincides with the natural action
of $\sigma$ induced by permuting the factors. To show that two given elements
of $\pi_1^{et}(\sX,x)$ coincide, it is enough to show that their actions
on the geometric fibers over $x$ of Galois \'etale covers $\sY\,\longrightarrow
\, \sX$ agree. In fact it suffices to consider only Galois \'etale covers
$$
\widetilde{f}\,:\, \sY \,=\, Y^n\,\longrightarrow\, \sX
$$
which are compositions of the form
$$
Y^n\, \stackrel{(f,\cdots ,f)}{\longrightarrow}\, X^n\, \longrightarrow\, \sX\, ,
$$ 
where $f\,:\,Y\,\longrightarrow\, X$ is a Galois \'etale cover, because
these covers $\widetilde{f}$ are cofinal in
the category of Galois \'etale covers of $\sX$; here $X^n \,\longrightarrow\, \sX$
is the quotient map. Let
$f^{-1}(x)$ denote the fiber of $f$ over the point $x$. Then $\widetilde{f}^{-1}(x)\,
=\,f^{-1}(x)^n$. We note that the action of $S^n$ via $\theta$ on $f^{-1}(x)^n$ is
simply the natural permutation action. Indeed, this can be deduced from the fact that the
following diagram is Cartesian 
$$\xymatrix{ 
[f^{-1}(x)^n/S^n]\ar[r]\ar[d] & Y^n \ar[d]^{\widetilde{f}} \\
[x/S^n]\ar[r] & \sX
}$$ 
where $[f^{-1}(x)^n/S^n]$ is the stack quotient by the permutation action of
$S^n$ on $f^{-1}(x)^n$. Now if  $$ \widetilde{\phi}=(\phi_1,\cdots ,\phi_n)\,\in\,
\Gal(Y^n/X^n)\, , $$
then for $\sigma \,\in\, S^n$, the equality 
$$\theta(\sigma)\cdot \widetilde{\phi}\cdot \theta(\sigma)^{-1}\,=\,
(\phi_{\sigma^{-1}(1)},\cdots ,\phi_{\sigma^{-1}(n)})$$
can be deduced by comparing their actions on $f^{-1}(x)^n$. This proves the claim that
the action of $S^n$ on $\pi_1^{et}(X^n,x)$ is the same as the one induced by the
permutation action of $S^n$ on $X^n$.
\end{proof}

\begin{lemma}\label{refinedav}
Let $\sX/k$ be a separated Deligne--Mumford stack over a noetherian base scheme and
$\pi\,:\,\sX\,\longrightarrow\, X$ be its coarse moduli space. Then \'etale locally
around any point $p$ of $X$, the stack $\sX$ can be written as a quotient of a $k$-scheme
by the isotropy group of $\sX$ at $p$. 
\end{lemma}

\begin{proof}
By \cite[Lemma 2.2.3]{av}, we may assume that $\sX\,=\,[V/G]$, where $G$ is a finite
group acting on a scheme $V$. Let $\widetilde{p}$ be a point of $V$ lying over $p$, and
let $G_{\widetilde{p}}\subset G$ be the isotropy group of $\widetilde{p}$. The group
$G_{\widetilde{p}}$ is canonically identified with the isotropy group $G_p$ of
the stack $\sX$ at the point $p$.

We claim that there exists a Zariski open neighborhood of $\widetilde{p}$ in $V$, such 
that the isotropy group of all points in this neighborhood is contained in 
$G_{\widetilde{p}}$.

To prove the above claim, note that for any $g\in G$, the fixed locus $Z_g$ of $g$ 
is a closed subscheme of $V$. We now simply take this neighborhood to be the complement 
of $\bigcup_{g\in G_{\widetilde{p}}} Z_g$.

We now replace $V$ by this neighborhood and assume that $V$ itself has the  property that
all isotropy groups are contained in $G_{\widetilde{p}}$. This ensures that the map
$[V/G_{\widetilde{p}}]\,\longrightarrow\, [V/G]\,=\,\sX$, as well as the induced map on
coarse moduli spaces is finite
\'etale. Thus by Remark \ref{coarsenoohi}, the stack $[V/G_{\widetilde{p}}]$
is pullback of a finite \'etale cover of the coarse moduli space of $\sX$. In other words,
\'etale locally on the coarse moduli space, $\sX$ of the type $[V/G_{\widetilde{p}}]$, as
required. 
\end{proof}

\begin{notation}[Specialization map on isotropy groups]\label{notspe}
Let $\sX/k$ be a separated Deligne--Mumford stack. Then for any geometric point
$q$ of $\sX$, specializing to a geometric point $p$,
one can produce a homomorphism
\begin{equation}\label{qp}
\phi_{q,p}\,:\,G_q\,\longrightarrow\, G_p
\end{equation}
which is well defined up to conjugation by an element of $G_p$;
we will recall its construction. To define $\phi_{q,p}$,
we may work \'etale locally on the coarse moduli space of $\sX$ and by
Lemma \ref{refinedav} assume that $\sX$ is the quotient $[V/G_p]$, where $G_p$ acts on $V$
with a fixed point $\widetilde{p}$ lying over $p$. Since the point $q$ of $\sX$ specializes
to $p$, there exists a point $\widetilde{q}$ (not necessarily unique) of $V$, lying over
$q$ which specializes to $\widetilde{p}$. As $\widetilde{q}$ specializes to
$\widetilde{p}$, the isotropy group $G_{\widetilde{q}}$ is a subgroup of the isotropy group
$G_{\widetilde{p}}$ of $p$. However, we also have natural identifications
$G_{\widetilde{p}}\,=\, G_p$ and $G_{\widetilde{q}}\,=\, G_q$. This gives rise to an
inclusion $\phi_{q,p}\,:\,G_q\, \longrightarrow\, G_p$. Note however, a different choice
of $\widetilde{q}$ leads to another homomorphism $G_q\,\longrightarrow\, G_p$ which
differs from the above choice by an inner automorphism of $G_p$. Thus $\phi_{q,p}$ is
well defined up to an inner automorphism. 
\end{notation}

\begin{lemma}\label{specialisation} Let $p,q$ be two geometric points of a
connected separated
Deligne--Mumford stack $\sX$ with $q$ specializing to $p$. Let $x$ be any other
geometric point of $\sX$. Then the following diagram commutes up to an
inner automorphism of $\pi_1^{et}(\sX,x)$:
$$\xymatrix{
G_q \ar[d]_{\phi_{q,p}}\ar[rr]^{\phi_q^x\ \ \ \ } & & \pi_1^{et}(\sX,x) \\
G_p \ar[rru]_{\phi_p^x} &  &
}$$
(see \eqref{qp} and \eqref{re1} for the homomorphisms).
\end{lemma}

\begin{proof}
Without loss of generality we may assume that $x\,=\,p$ and hence $\phi_p^x\,=\,\phi_p$,
where $\phi_p$ is the canonical homomorphism from $G_p$ to $\pi_1^{et}(\sX,p)$
induced by $$[p/G_p] \,\longrightarrow\, \sX\, .$$
Let $V\, ,\widetilde{p}$ and $\widetilde{q}$ be as in Notation \ref{notspe}. We denote by
$\widetilde{V}$ the strict henselization of $V$ at $\widetilde{p}$ and lift points
$\widetilde{p}$ and $\widetilde{q}$ to points of $\widetilde{V}$, which are also denoted
by $\widetilde{p}, \widetilde{q}$ for ease of notation. The morphisms $[p/G_p]\,
\longrightarrow\, \sX$ and $[q/G_q]\,\longrightarrow\, \sX$ factor through 
$[\widetilde{V}/G_p]$. Therefore, to prove the lemma, we may replace $\sX$ by the stack
$[\widetilde{V}/G_q]$. Since $\widetilde{V}$ is strict henselian, the homomorphism
$$G_p \,\longrightarrow\, \pi_1^{et}([\widetilde{V}/G_p],p)$$
is a bijection. Thus it remains to show that the homomorphism
$$ \phi_q^p\,:\, G_q \,\longrightarrow\, \pi_1^{et}([\widetilde{V}/G_p],p)$$ defined in
\eqref{re1}, coincides, up to an inner automorphism, with the inclusion of
isotropy groups $G_q \,\inj\, G_p$. However, this follows easily by comparing the
actions of $G_q$ and $G_p$ on the geometric fibers of the universal cover
$\widetilde{V}\,\longrightarrow\, [\widetilde{V}/G_p]$. 
\end{proof}

\begin{proof}[Proof of Theorem \ref{sym}]
Let $\sX\,=\,[X^n/S^n]$, so that
$\pi\,:\,\sX\,\longrightarrow\, {\rm Sym}^n(X)$ is its coarse moduli space. Let
$x$ and $x_0$ be as in the statement of the theorem. We think of $x$ also as a
geometric point of both $\sX$ and $X^n$. The quotient map 
$$ f\,:\,X^n\,\longrightarrow\, \sX$$ 
is a Galois \'etale cover with Galois group $S^n$. By Lemma \ref{sn}, we have 
$$ \pi_1^{et}(\sX,x) \,=\, \pi_1^{et}(X,x_0)^n\rtimes S^n\, .$$
Moreover, the map 
$$ [x/S^n] \,\longrightarrow\, \sX$$ 
induces a section $t\,:\, S^n\,\longrightarrow\, \pi_1^{et}(\sX,x)$ of
$\pi_1^{et}(\sX,x)\,\longrightarrow\,
S^n$. The image $t(S^n)$ coincides with the image of the homomorphism
$$ G_x \,=\, S^n \,\longrightarrow\, \pi_1^{et}(\sX,x) $$
in \eqref{gq}.

Using notation from Section \ref{sec2}, we claim that the closed
normal subgroup of $\pi_1^{et}(\sX,x)$ generated by image of $t$ is
entire $N_x$ (see Theorem \ref{noohi}). By definition of $N_x$, to prove the claim
it is enough to show 
that given any other point $p$, the image of $$ G_p\,\longrightarrow\,
\pi_1^{et}(\sX,x)$$ 
is contained inside the image of $t$. Now, lift $p$ to a geometric point of 
$X^n$, and let $I\,\subset\, \{1,\cdots ,n\}$ be the maximal subset such that 
the coordinates of $p$ indexed by $I$ are equal. We have the subvariety 
$X^n_I\,\subset\, X^n$ defined by the condition that all coordinates indexed by 
$I$ are equal. Let $q$ be the generic point of $X^n_I$. Then $q$ 
specializes to $p$ and the inclusion $G_q\,\inj\, G_p$ is an isomorphism. But 
$q$ also specializes to $x$. The theorem now follows from 
Lemma \ref{specialisation}.
\end{proof}

\end{document}